\documentclass{scrartcl}

\newcommand{\Description}[1]{}

\usepackage{amsmath,amssymb,amsfonts,amsthm}
\usepackage{mathtools}
\usepackage{csquotes}
\usepackage[inline]{enumitem}
\usepackage{tikz}
    \usepackage{tkz-berge}
    \usetikzlibrary{calc}
    \usetikzlibrary{intersections}
    \usetikzlibrary{arrows, positioning, automata}
\usepackage{pgfplots}
    \pgfplotsset{compat=1.3}
\usepackage{longtable}
\usepackage{mathtools}
    \mathtoolsset{centercolon}
\usepackage{multicol}
\usepackage{setspace}
\usepackage{textcomp}
\usepackage{upquote}
\usepackage{xtab}
\usepackage{xparse}
\usepackage{soul}
\usepackage{xspace}

\newtheorem{theorem}{Theorem}[section]
\newtheorem{lemma}[theorem]{Lemma}

\theoremstyle{definition}

\newtheorem{definition}[theorem]{Definition}
\newtheorem{example}[theorem]{Example}
\newtheorem{remark}[theorem]{Remark}

\newtheorem{algorithm_plain}[theorem]{Algorithm}

\usepackage[pdfborder={0 0 0}]{hyperref}

\newcommand{\Vclosed}{$V\!$-closed\xspace}
\newcommand{\generatorSet}{\mathbf{G}}
\newcommand{\generator}{\mathbf{g}}

\newcommand{\bT}{{\mathbf{T}}}

\newcommand{\bt}{{\mathbf{t}}}
\newcommand{\bj}{\boldsymbol{j}}

\newcommand{\Id}{\operatorname{I}}
\newcommand{\cA}{{\mathcal{A}}}

\newcommand{\cJ}{{\mathcal{J}}}
\newcommand{\cL}{{\mathcal{L}}}

\newcommand{\norm}[1]{\left\|#1\right\|}
\newcommand{\abs}[1]{\left|#1\right|}
\newcommand{\set}[1]{\left\lbrace #1 \right\rbrace}

\newcommand{\minimize}{\operatorname{minimize}}
\newcommand{\vardot}{\mathord{\,\cdot\,}}
\newcommand{\JSR}{\operatorname{JSR}}
\newcommand{\co}{\operatorname{co}}
\newcommand{\coell}{\operatorname{co}_{\operatorname{e}}}

\newcommand{\closure}{\operatorname{cl}}

\DeclareMathOperator*{\minp}{min\vphantom{p}}
\DeclareMathOperator*{\maxp}{max\vphantom{p}}
\newcommand{\diag}[1]{\operatorname{diag}\left(#1\right)}
\newcommand{\prho}{\rho_{\varepsilon,\norm{\vardot}}}

\renewcommand{\Re}{\operatorname{\mathfrak{Re}}}
\renewcommand{\Im}{\operatorname{\mathfrak{Im}}}

\newcommand{\CC}{{\mathbb C}}
\newcommand{\NN}{{\mathbb N}}
\newcommand{\RR}{{\mathbb R}}

\newcommand{\tbmatrix}[2][r]{\left[\begin{matrix*}[#1]#2\end{matrix*}\right]}

\newcommand{\titlevar}{%
    A Hybrid Approach to Joint Spectral Radius Computation%
}
\newcommand{\abstractvar}{%
    In this paper we propose a new method to determine the joint spectral 
    radius of a finite set of real matrices by verifying that a given family 
    of candidates actually consists of spectrum maximizing products.
    Our algorithm aims at constructing a finite set-valued tree according to 
    the approach of Möller and Reif using a norm that is constructed in the 
    spirit of the invariant polytope algorithm. This combines the 
    broad range of applicability of the first algorithm with the efficiency 
    of the latter.
}

\date{}

\author{Thomas~Mejstrik%
\thanks{University of Vienna, Austria,
{e-mail: \texttt{\small thomas.mejstrik@gmx.at}}}
\and
Ulrich~Reif%
\thanks{Technische Universität Darmstadt, Germany,
{e-mail: \texttt{\small reif@mathematik.tu-darmstadt.de}}}
}

\title{\titlevar}

\begin{document}
\maketitle

\begin{abstract}
\abstractvar
\end{abstract}


\section{Introduction}

We are concerned with the exact computation of the 
\emph{joint spectral radius} (JSR) of a finite set 
$\cA = \{A_1,\dots,A_J\}$ of real matrices $A_j \in \RR^{s\times s}$.
The JSR describes
the maximal asymptotic growth rate of the norms of products of matrices from that set.

The $\JSR$ was initially introduced by Rota and Strang in 1960~\cite{RS1960},
and has since found applications in various seemingly unrelated fields of mathematics.
For instance, in computing the regularity of wavelets and subdivision schemes~\cite{DL1992},
in analysing code capacities~\cite{MOS2001},
in studying the stability of linear switched systems~\cite{Gur95},
and in exploring connections with the Euler partition function~\cite{Prot00},
see also~\cite{GP2016,Jung2009} and the references therein.

Basically, except for the treatment of special cases, there exist two 
different approaches to the computation of the JSR: The finite tree
algorithm of Möller and Reif~\cite{MR2014} and the invariant polytope 
algorithm~\cite{GP2013,GP2016,Mej2020}. 
Both algorithms start from a given set of matrices that are expected to 
be spectrum maximizing products. When either algorithm terminates, this 
assumption is verified, and therewith the JSR is determined. If not, 
runtime limits may have been exceeded or 
the given set of candidates was deficient. The latter can 
be caused by an insufficient search for spectrum maximizing products,
or it may even have been impossible if the given set~$\cA$ of matrices 
does not possess the finiteness property~\cite{LW1995,BM2002,HMST2011}.

Invariant polytope algorithms try to construct a finite 
set of points such that the image of any of these points under 
any of the given matrices lies inside the convex hull
of the points. Existing implementations have proven to be very efficient,
even for matrices of high dimension. However, the range of applicability
is restricted by the fact that termination of the algorithms is in general only possible 
when the given set~$\cA$ of matrices has a simple leading eigenvalue.
By contrast, the finite tree algorithm potentially works even in these cases.
The finite tree algorithm aims at constructing 
a finite tree with nodes consisting of sets of matrix products
whose leaves are bounded by~$1$ with respect to a given matrix norm.
A disadvantage of this algorithm is its strong dependency of the 
runtime on the chosen norm. 

In this paper, we suggest a hybrid method that promises to combine the 
broad range of applicability of the finite tree algorithm with the 
efficiency of the invariant polytope algorithm.
As we will show, the resulting \emph{adapted tree search} 
may converge even if the invariant polytope algorithm does not.

The paper is organized as follows: In the next section, we introduce notation
and recall some important facts about the joint spectral radius 
and its determination from the literature. In Section~\ref{sec:algorithm},
we establish adapted tree search as a new algorithm for JSR determination 
and analyse its basic properties, and in Section~\ref{sec:implementation}, we describe
implementation details and discuss an example.

\subsection{Definitions and basic properties}
We denote
the set of positive integers by $\NN$, 
the set of non-negative integers by $\NN_0$,
the spectral radius of the square matrix $A$ by $\rho(A)$,
the identity matrix by $\Id$,
the closure of the set $X\subseteq\CC^s$ by $\closure(X)$,
the convex hull of $X$ by $\co(X)$, 
and the symmetric convex hull of $X$ by $\co_s X =\co(X \cup -X)$.
Sums and products of sets are understood element-wise.

Throughout, $\cA = \{A_1,\dots,A_J\}$ denotes a
set of $J \in \NN$ matrices $A_j \in \RR^{s \times s}$, and 
$\cJ_n := \{1,\dots,J\}^n$ is the set of index vectors
with $n \in \NN_0$ elements.
Products of matrices are denoted by 
\begin{equation*}
A_{\bj} = A_{j_n} \cdots A_{j_1}
,\quad
\bj = [j_1, \ldots, j_n] \in \cJ_n
.
\end{equation*}
For $n=0$, the expression above is understood to be the identity matrix.
The length of the vector $\bj \in \cJ_n$ is denoted by $|\bj| = n$.
\begin{definition}
The \emph{joint spectral radius} (JSR) of~$\cA$ is defined by
\begin{equation}\label{equ_jsr}
  \JSR(\cA) =
  \lim_{n\rightarrow \infty}
  \max_{\bj\in\cJ_n}
  \norm{A_{\bj}}^{1/n},
\end{equation}
where $\norm{\vardot}$ is any sub-multiplicative matrix norm.
\end{definition}


We will discuss the ideas behind the aforementioned algorithms after defining some crucial concepts.

\begin{definition}
If the matrices $A_j\in\cA$ do not share a common invariant subspace
other than $\set{0}$ and $\RR^s$,
then the set~$\cA$ is called \emph{irreducible}.
\end{definition}


From an algebraic point of view, irreducibility expresses that there does not exist a 
change of basis under which all matrices of the given set are simultaneously block triangularized.

\begin{theorem}[{\cite[Proposition 3]{BW1992}}]
\label{thm:bounded} 
Given $v\in\RR^s \setminus \{0\}$, define the set $P(v)$ by
\begin{equation}
\label{equ_bounded}
  P(v) = \co_s \bigcup_{n\in\NN_0} \bigcup_{\bj \in \cJ_n} A_{\bj} v.
\end{equation}
If~$\cA$ is irreducible and $\JSR(\cA)=1$, 
then $P(v)$ is a bounded subset of $\RR^s$ with 
non-empty interior.
\end{theorem}

\begin{definition}
\label{def:minknorm}
Let $P\in\RR^s$ be a bounded, convex set with non-empty interior,
and such that
$t P\subseteq P$ for all $t \in [-1,1]$. 
We define the \emph{Minkowski norm}
$\norm{\vardot}_P:\RR^s\rightarrow\RR$ by
\begin{equation}
\label{equ_minknorm}
  \norm{x}_{P} = \inf \set{t>0:x\in tP}
  .
\end{equation}
\end{definition}

\begin{lemma}[{\cite[Theorem~2.2]{Jung2009}}]
If~$\cA$ is irreducible, then
\begin{equation}\label{equ_norm}
  \norm{A_j x}_{P(v)} \leq \JSR(\cA) \cdot \norm{x}_{P(v)}
\end{equation}
for all $x\in\RR^s$, $v \in \RR^s\setminus\{0\}$, and $A_j\in\cA$.
\end{lemma}

\begin{definition}
The set~$\cA$ is said to possess the
\emph{finiteness property} if there exists a finite product 
$\Pi=A_{\generator}$, $\generator \in \cJ_n$,
such that $\rho(\Pi)^{1/n}=\JSR(\cA)$.
We will call any such product a \emph{spectrum maximizing product} or short an \emph{s.m.p.}.
\end{definition}

\begin{definition}
Let $\Pi\in\RR^{s\times s}$ be a matrix 
with eigenvalues sorted by modulus, i.e.~$|\lambda_1| \ge \cdots \ge |\lambda_s|$.
The eigenvalues $\lambda_i$ with $|\lambda_i| = \rho(\Pi)$ are the \emph{leading eigenvalues} 
and the corresponding eigenvectors are the \emph{leading eigenvectors}.
If $|\lambda_2| < \rho(\Pi)$ or $\lambda_1 = \bar{\lambda}_2 \not\in\RR$ and 
$|\lambda_3| < \rho(\Pi)$,
we say that the leading eigenvalue $\lambda_1$ is~\emph{simple}\footnote{%
In the literature a simple leading eigenvector is also called \emph{unique}.%
}.
\end{definition}
An immediate consequence of Theorem~\ref{thm:bounded} is the following:
\begin{lemma}
For any s.m.p.~$\Pi$ of an irreducible set~$\cA$ 
the geometric and algebraic multiplicities of the leading eigenvectors coincide.
\end{lemma}


\begin{definition}
The set~$\cA$ is said to possess a \emph{spectral gap} 
if there exists $\gamma < 1$ such that any
product $A_{\bj}$ is either an s.m.p.~or its spectral radius is bounded by
\begin{equation*}
  \rho(A_{\bj}) \le \gamma \cdot \JSR(\cA)^{|\bj|}
  .
\end{equation*}
\end{definition}


\subsection{Preparation}
The invariant polytope algorithms and the tree algorithm
share the following basic procedure:
First, an \emph{s.m.p.-candidate} $\Pi = A_{\generator}, \generator \in \cJ_n$, 
which is a product for which there is numerical evidence that it is in fact an s.m.p.,
is determined; Actually, a whole list of candidates is determined, but this is irrelevant
at this point.
A fast algorithm identifying s.m.p.-candidates can be found in~\cite[Section~3]{Mej2020}.
Second, scaling the given matrices using the factor
$\lambda := \rho(\Pi)^{1/n}$ yields the set 
$\tilde \cA := \{\lambda^{-1} A_1,\dots,\lambda^{-1} A_J\}$.
The corresponding s.m.p.-candidate $\tilde \Pi = \lambda^{-n} \Pi$,
has spectral radius $\rho(\tilde \Pi) = 1$.
Third, and this is the crucial task,
one tries to prove that $\tilde \Pi$ is actually 
an s.m.p.~of $\tilde \cA$, which is equivalent to
$\JSR(\tilde \cA) = 1$ and $\JSR(\cA) = \lambda$.

From now on, we assume that the first two steps have already 
been carried out so that, omitting notation of the tilde, 
all s.m.p.-candidates $\Pi_1,\dots,\Pi_M$
have spectral radius~$1$, and the conjecture to be proven is $\JSR(\cA) = 1$.

\subsection{Two known algorithms}
Let us briefly recall the two general-purpose algorithms for JSR determination
described in the literature:
The \emph{finite tree algorithm}~\cite{MR2014}
and the \emph{invariant polytope algorithm}~\cite{GP2013,GP2016,Mej2020}.

Apart from the s.m.p.-candidates $\Pi_1,\dots,\Pi_M$, 
the \emph{finite tree algorithm} needs a fixed sub-multiplicative matrix norm as an additional input.
It then tries to compute a tree where each node consists of a set of matrix products.
Children of nodes consist of all matrix products which can be formed by multiplying the parent 
matrices with a single matrix~$A_j$ or with all integer powers of an s.m.p.~$\Pi_m$.
A leaf is a node in the tree with the property that all its containing matrix products 
have norm not exceeding~$1$.

The \emph{invariant polytope algorithm}
on the other hand tries to construct a vector-norm under
which induced matrix norm all matrices from the set~$\cA$ have norm not exceeding~$1$.
It does so by trying to construct the unit ball $P$ of a norm,
which is invariant under all matrices in~$\cA$.
The unit ball is constructed in the fashion of~\eqref{equ_bounded}
with a carefully chosen starting vector $v$
(or actually a set of starting vectors).

\section{Adapted tree search}
\label{sec:algorithm}
In order to combine the ideas of the two algorithms described above,
we introduce the basic concept of the finite tree algorithm using notation
mostly following~\cite{MR2014}.

The set $\bigcup_{n \in \NN_0}\cJ_n$ of index vectors of arbitrary length,
and with it the set of all finite matrix products $A_{\bj} = A_{j_n} \cdots A_{j_1}$,
will be equipped with the structure of a rooted \mbox{$J$-ary} tree in a
natural way.
To prove $\JSR(\cA) \le 1$, it is sufficient to determine a
finite sub-tree with the property that 
every leaf contains a matrix with norm not exceeding~$1$.
By contrast, if $\JSR(\cA) = 1$, then such sub-trees
are typically infinite, what makes them inconvenient for an 
immediate algorithmic scanning. To account for that situation, the tree algorithm 
aims at constructing a \emph{finite} tree whose nodes do not contain single matrices, but sets
of products, which can be finite or countably infinite.

\begin{definition}[$(\cA,\generatorSet)$-tree]
\label{def:tree}
Let $\generatorSet = \{\generator_1,\dots,\generator_I\}$ be a set of index vectors $\generator_i \in \bigcup_{n \in \NN}\cJ_n$,
whose corresponding matrix products $A_{\generator_i}$ have spectral radius not exceeding~1.
A finite, rooted tree $\bT$ with nodes consisting of sets of matrices
is called an \emph{$(\cA,\generatorSet)$-tree} if it has the following properties:
\begin{itemize}
\item The root node $\bt_0 := \{\Id\}$ contains the identity matrix.
\item Each node $\bt \in \bT$ is either 
\begin{itemize} 
\item a \emph{leaf}, i.e., it has no children,
\item or parent of exactly $J$ \emph{siblings} $\{A_j P : P \in \bt\}$, $j = 1,\dots,J$,
\item or parent of arbitrary many \emph{generators} $\{A_\generator^n P:P \in \bt,\ n \in \NN_0\}$ 
for some $\generator \in \generatorSet$.
\end{itemize}
\item No generator is a leaf.
\end{itemize}
A leaf is called \emph{covered} if it is a subset of one of its predecessors,
and otherwise \emph{uncovered}.
The \emph{leafage} of $\bT$ is denoted by 
$\cL := \{L \in \bt : \text{ $\bt$ is uncovered leaf of $\bT$} \}$.
\end{definition}

\begin{definition}[1-boundedness]\label{def:1bd}
Given a sub-multiplicative matrix norm  $\norm{\vardot}$,
$\bT$ is called \emph{1-bounded} if $\sup_{L\in\cL} \norm{L} \le 1$
and \emph{strongly~1-bounded} if $\sup_{L\in\cL} \norm{L} < 1$.
\end{definition}

\begin{figure}
  \centering
    \includegraphics[height=5cm]{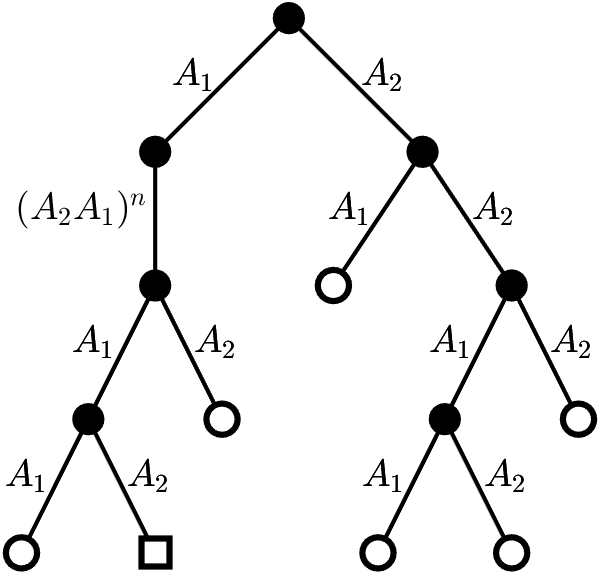}
    \caption{$(\cA,\generatorSet)$-tree for $\cA = \{A_1,A_2\}$ and $\generatorSet = \{[1,2]\}$
    with covered ($\square$) and uncovered ($\circ$) leafs.
    The covered leaf $\{A_2 A_1 (A_2 A_1)^n A_1, n \in \NN_0\} = \{(A_2 A_1)^n A_1: n \in \NN\}$
    is a (proper) subset of its grandparent $\{(A_2 A_1)^n A_1: n \in \NN_0\}$.
    }
    \label{fig:examplary_tree}
\end{figure}

In practice, the matrices $A_\generator$, $\generator \in \generatorSet$,
are s.m.p.-candidates
or matrices with spectral radius close to~$1$.
Figure~\ref{fig:examplary_tree} illustrates the definition, and in particular 
the emergence of covered nodes.

The main result in \cite{MR2014} (Theorem 3.3) is the following:
\begin{theorem}
\label{thm:FT}
Let $\bT$ be an $(\cA,\generatorSet)$-tree and 
$\norm{\vardot}$ be a sub-multiplicative matrix norm.
If $\bT$ is 1-bounded, then $\JSR(\cA) \le 1$.
\end{theorem}

We next give a definition and a theorem which are at the core of 
invariant polytope algorithms:
\begin{definition}[\Vclosed{}ness]
Let $V\subset\RR^s$ be a finite set of vectors spanning $\RR^s$.
An $(\cA,\generatorSet)$-tree $\bT$ is said to be \Vclosed if
$\cL V\subseteq \co_s V$.
If there exists $\gamma <1$ such that $\cL V \in \gamma \co_s (V)$,
then $\bT$ is  \emph{strongly \Vclosed}.

\end{definition}

As in Definition~\ref{def:1bd}, we require that only matrices in the leafage of $\bT$ map vertices to interior of $\co_s V$.
Indeed, any leading eigenvector of an s.m.p.~will always get mapped to the boundary of $\co_s V$
by matrices element  of a covered node.

%

The conditions under
which the invariant polytope algorithm terminates~\cite[Theorem 3.3]{GP2016}
can be restated using the terminology of the tree algorithm.
\begin{theorem}
\label{thm:ipa}
If the set~$\cA$ is irreducible, and
its $\JSR$ does not exceed~$1$,
then the following are equivalent:
\begin{itemize}
\item There exists an $(\cA,\generatorSet)$-tree with exactly 1 generator directly located at the root node,
which is strongly $V\!$-closed for a set $V$ of a leading eigenvector $v_1$ of an s.m.p.\ %
and suitable scaled additional vectors $v_2,\ldots,v_s$, such
that $\set{v_1, \ldots, v_s}$ span $\RR^s$.

\item The set~$\cA$ has a spectral gap,
there exist exactly 1 s.m.p.~(up to cyclic permutations and powers),
and the s.m.p.~has exactly one simple leading eigenvalue.
\end{itemize}
\end{theorem}

Strict closedness is relevant in applications for numerical reasons.
Deciding if a vertex $Lv$  is inside or outside of a polytope $\co_s(V)$
is hard or even impossible
for a point on the boundary when working in a  floating-point environment
because even smallest rounding errors or inexact input data may distort the result.

Invariant polytope algorithms are known to be highly efficient in applications,
while the finite tree approach covers a wider range of treatable cases.
The following considerations prepare a hybrid method
that is designed to combine the advantages.


%
\begin{theorem}
\label{thm:FTIP}
If there exists a \Vclosed $(\cA,\generatorSet)$-tree, then $\JSR(\cA) \le 1$. 
\end{theorem}
\begin{proof}
Since $V=\set{v_1,\ldots,v_m}$ is spanning $\RR^s$, 
the set $P := \co_s(V)$ has non-empty interior and hence 
defines a Minkowski norm $\norm{\vardot}_P$ according to 
Definition~\ref{def:minknorm}.
Moreover,
any point $x \in \RR^s$ with $\norm{x}_P \leq 1$
can be written in the form $x = \sum_{m=1}^M \lambda_m v_m$
with coefficients  $\lambda_m$ satisfying
$\sum_{m=1}^M \abs{\lambda_m} \leq 1$.
Let $L \in \cL$ be any matrix in the leafage of $\bT$. 
We obtain $Lx = \sum_{m=1}^M \lambda_m Lv_m$ and observe that 
$L v_m \in P$ by assumption. This implies
$Lx \in \co_sP = P$ and hence $\norm{Lx}_P \le 1$. We conclude 
$\norm{L}_P = \max\{\norm{Lx}_P : \norm{x}_P=1\}\le 1$,
and Theorem~\ref{thm:FT} applies.
\end{proof}


Theorem~\ref{thm:ipa} informs us that,
whenever the invariant polytope algorithm terminates,
there exists a strongly \Vclosed tree
of a a very special shape.
But in addition, \Vclosed and even strongly \Vclosed trees may exist
in cases where invariant polytope algorithms fail.
\begin{theorem}
\label{thm:terminate}
Let $V$ be a finite set of vectors spanning $\RR^s$.
If $\cA V \subseteq \co_s V$, then~$\cA$ possesses a \Vclosed $(\cA,\generatorSet)$-tree.
Moreover, there are cases where invariant polytope algorithms do not terminate, 
but for which strongly \Vclosed trees exist nevertheless.
\end{theorem}

\begin{proof}
The first part of the theorem follows easily by 
choosing the tree $\bT$ consisting of nothing but the root and the matrices
$A_1,\dots,A_J$ as its children. In this case, closedness of $V$ with respect 
to~$\cA$ is equivalent to $\bT$ being a \Vclosed tree.

The
second part of the
theorem is established by means of a simple example:
Consider the pair
\begin{equation*}
   A_1 = \begin{bmatrix}
   1 & 0 \\ 0 & -1    
   \end{bmatrix}
   ,\quad 
   A_2 = \begin{bmatrix}
   0 & 1 \\ 1/2 & 0    
   \end{bmatrix},
\end{equation*}
and the s.m.p.-candidate $\Pi = A_1$
corresponding to the index set $\generatorSet = \set{\tbmatrix{1}}$.
According to Theorem~\ref{thm:ipa},
the two leading eigenvalues $\pm 1$ of $\Pi$ rule out the
possibility that the invariant polytope algorithm terminates.
However, choosing the leading eigenvectors $v_1 = \tbmatrix{1& 0}^T$ and $v_2 = \frac{1}{2}\tbmatrix{0 & 1}^T$ of $\Pi$
to form the set $V = \set{v_1,v_2}$,
the $(\cA,\generatorSet)$-tree $\bT_1$ consisting of the root and the leafs $\set{A_1}$ and $\set{A_2}$
is \Vclosed since
\begin{equation*}
  A_1 v_1 = v_1
  ,\quad 
  A_1 v_2 = -v_2
  ,\quad 
  A_2 v_1 = v_2
  ,\quad 
  A_2 v_2 = v_1/2
\end{equation*}
all lie in $\co_s(V)$.
This confirms the claim $\JSR(\cA) = 1$.
However, $\bT_1$ is not a strongly \Vclosed tree because $A_1v_1$, $A_1 v_2$, and $A_2 v_1$ all lie on the boundary of $\co_s(V)$.
An alternative tree $\bT_\gamma$,
with the set 
$V=\set{v_1, v_2}$ consisting of the leading eigenvectors
$v_1 = \tbmatrix{1 & 0}^T$ and $v_2 = \frac{3}{4}\tbmatrix{0 & 1}^T$,
is shown in Figure~\ref{fig:2}.
It contains one uncovered leaf,
\begin{equation*}
  \cL = \{A_2 A_1^n : n \in \NN\}
  = \left\{ \begin{bmatrix}
  0 & -1^{n} \\ 1/2 & 0
  \end{bmatrix}\right\},
\end{equation*}
which has to be checked for closedness.
For $L \in \cL$, both
$L v_1 = \frac{1}{2}\tbmatrix{0&1}^T$
and
$L v_2 = \pm\frac{3}{4}\tbmatrix{1&0}^T$
lie in $\frac{3}{4}\co_S(V)$. Hence, $\bT_\gamma$ is a strongly \Vclosed tree.
\end{proof}
\begin{figure}
    \centering
    \raisebox{1.4cm}{%
    \includegraphics[height=1.6cm]{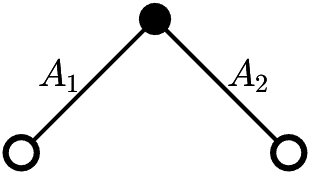}}
    \hspace{25mm}
    \includegraphics[height=3.0cm]{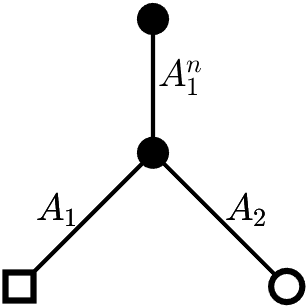}
    \caption{\Vclosed tree $\bT_1$ \emph{(left)} and
    strongly \Vclosed tree $\bT_\gamma$ \emph{(right)}
    for the example in the proof of Theorem~\ref{thm:terminate}.
    }
    \label{fig:2}
\end{figure}

Unlike in the most simple example presented in the proof,
single children and their descendants are typically countably infinite.
But even then, ranges for the products $Lv$ can be 
determined by numerical or analytical methods which are sufficiently tight to establish 
\Vclosed{}ness. We will discuss this issue in the following section.


\section{Algorithms}
We envision two ways how to combine the tree and the invariant polytope algorithms:

\subsection{Invariant polytope flavoured tree algorithm}
We start by running an invariant polytope algorithm for a few steps.
The so-far constructed norm is then already well adapted to the given 
task, meaning that the norms of the matrices $A_j$ are close to~$1$.
After that, we we feed this norm into the tree algorithm.


\subsection{Tree flavoured invariant polytope algorithm}
Run an invariant polytope algorithm.
Instead of always testing whether the $J$ children of a vertex 
are mapped into the interior, we try to construct $(\cA,\generatorSet)$-subtrees
and check whether their leafage is $1$-bounded.


We give a brief pseudo code implementation of this algorithm.
The only difference to the original invariant polytope algorithm (see e.g.~\cite[Section 4]{Mej2020})
lies in the lines marked with $(\ast)$.
The original invariant polytope algorithm chooses in those lines as $\boldsymbol{T}$ always the set $\{A_1,\ldots,A_J\}$.

\label{sec:implementation}  
\begin{algorithm_plain}[Tree-flavoured-invariant-polytope-algorithm]
\label{alg_feta_flavoured_ipa}
~\begin{flushleft}
{\bfseries Given: } s.m.p.-candidates $\Pi_1,\ldots,\Pi_M$ with $\rho(\Pi_m)=1$, $m=1\ldots,M$\\
{\bfseries Result upon Termination: } Invariant polytope for~$\cA$
\rule{0.7\textwidth}{0.8pt}\\
Select leading eigenvectors $V \coloneqq \set{v_1,\,\ldots,\,v_M}$\\
Set $V_{\text{new}} \coloneqq V$\\
{\bfseries while} $V_{\text{new}}\neq\emptyset$\\
$\qquad$ Set $V_{\text{rem}}\coloneqq V_{\text{new}}$\\
$\qquad$ Set $V_{\text{new}}\coloneqq \emptyset$\\
$\qquad$ {\bfseries for} $v\in{\cA}V_{\text{rem}}$ {\bfseries do}\\
$\qquad\qquad$ Construct some $\bT$ satisfying Definition~\ref{def:tree}\hfill $(\ast)$ \qquad ~\\
$\qquad\qquad$  {\bfseries if}  $L v \notin \co_s(V)$ for any 
$L\in\cL(\bT)$ {\bfseries then} \hfill $(\ast)$ \qquad ~\\
$\qquad\qquad\qquad$ Set $V\coloneqq V\cup v$ \\
$\qquad\qquad\qquad$ Set $V_{\text{new}} \coloneqq V_{new} \cup v$ \\
{\bfseries return} $\co_s V$
\end{flushleft}
\end{algorithm_plain}

\subsection{Implementation details}

Now, we are concerned with the task
to check whether all points of a finitely expressible tree are mapped into a polytope.
Instead of using matrix balls, as done in~\cite{MR2014},
we will make use of the pseudo spectral radius of matrices in the following.

Let $\co_s P\varsubsetneq\RR^s$ be a polytope,
$w\in\RR^s$ a vector, and
$X,\Pi \in\RR^{s\times s}$ some matrices.
Furthermore, we presume that $X\Pi^\infty v\in P$,
where $\Pi^\infty$ denotes the set of accumulation points of $\Pi^{\ast}=\set{\Pi^n:n\in\NN_0}$.
Our task is to check whether
\begin{equation}
\label{equ_problem1}
X\Pi^{\ast} v\in \co_s P.
\end{equation}

Without loss of generality we can assume that $\rho(\Pi)=1$
and that the algebraic and geometric multiplicities coincide 
for all of its leading eigenvalues.

In general the set $\Pi^{\ast}$ has infinitely many limit points.
Denote with $\Pi^\circ$ one of them and choose $L\in\NN$.
If $\Pi^\infty$ is finite, 
then we can choose $L$ equal to the number of accumulation points of $\Pi^\infty$.

We are going to estimate the difference 
$\norm{X(\Pi^n-\Pi^l\Pi^\circ) v}$, $l=0, \ldots, \allowbreak L-1$,
by splitting up $v$ into the part which is in the subspace 
spanned by the leading eigenvectors of $\Pi$ and the rest.

Denote with $K$ the number of leading eigenvalues of $\Pi$.
From the Kreiss Matrix Theorem~\cite[Satz 4.1]{Kreiss1962} and Theorem~\ref{thm:bounded}
it follows that 
there exists an invertible matrix $V$ and a matrix $\Delta$ which
is a diagonal matrix $\Lambda^K$ 
with all entries equal to one in modulus for the first $K$ entries,
and an upper triangular matrix $T$ for the remaining part,
i.e.~%
\begin{equation}
\label{equ_pi}
\Pi = V \Delta V^{-1} 
\text{ with }
\Delta = \begin{bmatrix}
\Lambda^K & 0\\0 & T
\end{bmatrix},
\Lambda^{K} = \begin{bmatrix}
    \lambda_1 &        & 0           \\
                & \ddots &             \\
    0           &        & \lambda_K
    \end{bmatrix}.
\end{equation}
Now, with $q^n:=X\Pi^nv$ and $q^{l\circ}:=X\Pi^l\Pi^\circ v$,
\begin{equation}
\label{equ_difference}
q^n-q^{l\circ} = 
XV\Delta^nV^{-1}v - XV\Delta^l\Delta^{\!\circ} V^{-1}v =
XV(\Delta^n-\Delta^l\Delta^{\!\circ})V^{-1} v.
\end{equation}
We set $w=V^{-1} v$ and split up $w = w^K + w^R$ 
where $w^K$ holds the first $K$ entries of $w$ and $w^R$ the rest.
The difference $\Delta^{n,l\circ}=\Delta^n-\Delta^l \Delta^{\!\circ}$ takes the form
\begin{align*}
\Delta^{n,l\circ} &= 
\begin{bmatrix}
\Lambda^n - \Lambda^l\Lambda^{\!\circ} & 0 \\
0 & 0
\end{bmatrix} 
+
\begin{bmatrix}
0 & 0 \\
0 & T^n
\end{bmatrix}
=
\Delta^{n,l\circ,K} + \Delta^{n,R},
\end{align*}
where $\Delta^{n,l\circ,K}$ contains the first $K$ entries of the diagonal of $\Delta^{n,l\circ}$ 
and $\Delta^{n,R}$ the rest,
and $T$ is the same matrix as in~\eqref{equ_pi}.
Combining everything we get
\begin{equation*}
\norm{q^n - q^{l\circ}} \leq
\norm{X V} 
\left(
\norm{\Delta^{n,l\circ,K}} \cdot \norm{w^K} + \norm{\Delta^{n,R}}  \cdot \norm{w^R}
\right).
\end{equation*}

It remains to estimate, for arbitrary $n\in\NN$,
$(a)$ the norms of powers of the triangular matrix $\Delta^{n,R}$
whose entries on the main diagonal are all strictly less than one in modulus, and
$(b)$ the norms of $\Delta^{n,l\circ,K}$.

We estimate $(a)$ first.
Let $\gamma<1$ be the largest eigenvalue of $T$ in modulus.
By the Kreiss Matrix theorem there exist $M\geq 1$ and $1>\gamma_M>\gamma$, such that
\begin{equation*}
\norm{\Delta^{n,R}}\leq M \gamma_M^{n}.
\end{equation*}
One way to compute $M$ and $\gamma_M$ is to use the pseudo spectral radius~$\prho$~\cite{psapsr,Meng2006}.
For a given $\varepsilon>0$ and matrix norm $\norm{\vardot}$
the pseudo spectral radius $\prho$ is defined as
\begin{equation*}
\prho(\Delta) =
\sup \set{
\abs{z}:
z\in\CC,\ 
\big|\big|
{(zI-\Delta)^{-1}}
\big|\big|
> \varepsilon^{-1}
}.
\end{equation*}
and it holds that
$M = \varepsilon^{-1} \cdot \prho(\Delta)$ and
$\gamma_M = \prho(\Delta)$.
Because $\gamma_M<1$ there exists $N_{\Delta^{n,R}}\in\NN$ such that
\begin{equation*}
\max_{n\in\NN} \norm{\Delta^{n,R}} = \max_{1\leq n\leq N_{\Delta^{n,R}}} \norm{\Delta^{n,R}}.
\end{equation*}

We proceed with estimating $(b)$.
The norm $\norm{w^K}$ typically is not small. 
But, we can make the 
norm $\norm{\Delta^{n,l\circ,K}}$ arbitrarily small by taking $L$ large enough
and finding the optimal index $l$.
Computing this norm with respect to the polytope $\co_s P$ 
may be prohibitively expensive, and thus we resort in the following
to the 2-norm for its estimation.
The polytope-norm can be bounded by the 2-norm up to a factor of $C_P$
(see Lemma~\ref{thm:bound2norm} for details).
In the 2-norm, for given $n\in\NN$, the above problem reduces to computing the minimum 
(and corresponding argument~$l$) of
\begin{align*}
\norm{\Delta^{n,l\circ,K}}_2 & = 
 \minp_{0\leq l < L}\ %
\maxp_{1\leq k \leq K}\ %
\abs{
\lambda^l_k \lambda^{\!\circ} - \lambda_k^n
} \\
& \leq \sup_{\substack{x\in\CC^K\\\norm{x}_2=1}}\ %
 \minp_{0\leq l < L}\ %
 \maxp_{1\leq k \leq K}\ %
 \abs{
 \lambda^l_k - x
 } =: e_\Delta.
\end{align*}
%
Clearly, for any given $\varepsilon>0$ there exists $L_\varepsilon$ such that $e_\Delta<\varepsilon$.
Putting everything together, for arbitrary $\varepsilon>0$ and any $n,l\in\NN_0$,
\begin{equation}
\label{equ_estimate}
\begin{aligned}
\norm{q^n} &\leq 
\norm{q^n-q^{l\circ}}+\norm{q^{l\circ}} \\
& \leq
\norm{X V}
\left(
    C_P
    \varepsilon    
     \norm{w^K}  
    + 
    \max_{1\leq n\leq N_{\Delta^{n,R}}} \norm{\Delta^{n,R}}  \cdot \norm{w^R}
\right)
+
\max_{0\leq l < L_\varepsilon} \norm{q^{l\circ}},
\end{aligned}
\end{equation}
where the right hand side does not depend on $n$ and $l$ any more.

\begin{remark}
\label{rem_easycase}
If $\Pi^\infty$ is finite, i.e.\ all leading eigenvalues of $\Pi$ are roots of unity,
we can choose $\varepsilon=0$, $L_\varepsilon=\abs{\Pi^\infty}$,
and
Equation~\eqref{equ_estimate} simplifies to
\begin{equation}
\label{equ_estimate_simple}
\norm{q^n} \leq 
\norm{X V} \cdot
    \max_{1\leq n\leq N_{\Delta^{n,R}}} \norm{\Delta^{n,R}}  \cdot \norm{w^R}
+
\max_{0\leq l < \abs{\Pi^\infty}} \norm{q^{l\circ}},
\end{equation}
\end{remark}


We show the applicability of the described method on an example.
\begin{example}
\label{ex:infinitypath}
Let $\cA=\set{A_1,\,A_2}$, with
\begin{equation*}
A_1 = \frac{1}{\sqrt{13}}\begin{bmatrix}0&-1&\phantom{-}2\\1&\phantom{-}0&-1\\2&-2&\phantom{-}1\end{bmatrix},\quad
A_2 = \frac{1}{\sqrt{13}}
\begin{bmatrix}\phantom{-}1&-2&\phantom{-}2\\-2&\phantom{-}2&\phantom{-}1\\\phantom{-}2&\phantom{-}1&-2\end{bmatrix}.
\end{equation*}
The set~$\cA$ has joint spectral radius $\JSR(\cA)=\rho(A_2)=1$.
The s.m.p.~$\Pi=A_2$ has two leading eigenvalues $\pm 1$,
and thus, the invariant polytope algorithm cannot terminate~\cite[Theorem 4]{GP2013}. 
Let 
$v_0 = \begin{bmatrix}
3+\sqrt{13}&-4-\sqrt{13}&1
\end{bmatrix}^T$ be the 
leading eigenvector corresponding to $+1$.
We show that the tree-flavoured-invariant-polytope-algorithm terminates with the polytope 
$P = \co_s \{%
v_0,\,\allowbreak
A_1 v_0,\,\allowbreak
A_1 A_1 v_0,\,\allowbreak
A_2 A_1 v0,\,\allowbreak
A_1 A_2 A_1 v_0,\,\allowbreak
A_2 A_2 A_1 v0
\}$, 
and whose corresponding finitely expressible tree can be seen in Figure~\ref{fig:infinitypath}.

\begin{figure}
\centering
\includegraphics[height=6cm]{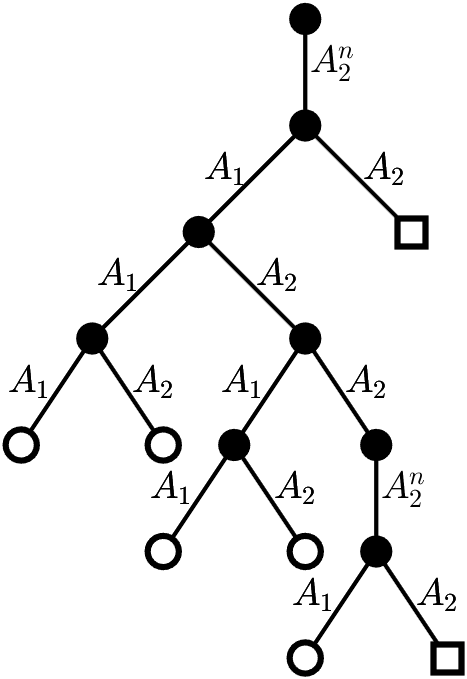}

\caption{%
Tree constructed by the tree-flavoured-invariant-polytope-algorithm
for the problem considered in Example~\ref{ex:infinitypath}.
Dots $(\bullet)$ mark the vertices of the polytope $P$;
Empty bullets $(\circ)$ mark vertices which are mapped into $\co_s P$;
Squares $(\square)$ mark covered nodes.
Note that the starting vector $v_0$ use to construct the polytope $P$
is an eigenvector of $A_2$,
and thus the node $\set{A_2^n v_0}$ only consists of one vector;
}
\Description{A finitely expressible tree for Example~\ref{ex:infinitypath}}
\label{fig:infinitypath}  
\end{figure}

With the notation from Remark~\ref{rem_easycase}, for 
$M_0=2$ we have that $\Pi^{M_0} = A_2^2$ has all eigenvalues equal to one.
We choose
$\Pi^\circ = \lim_{n\rightarrow\infty} (A^2)^n$.

After constructing the polytope $P$ consisting of the vertices 
in the dashed box 
{\small(%
$P = \co_s \{%
v_0,\,\allowbreak
A_1 v_0,\,\allowbreak
A_1 A_1 v_0,\,\allowbreak
A_2 A_1 v0,\,\allowbreak
A_1 A_2 A_1 v_0,\,\allowbreak
A_2 A_2 A_1 v0
\}$)},
the vertices denoted with $\star$
{\small(%
$\{%
A_1 A_1^2 v_0,\,\allowbreak
A_2 A_1^2 v_0,\,\allowbreak
A_1 A_1 A_2 A_1 v_0,\,\allowbreak
A_2 A_1 A_2 A_1 v_0
\}$)}
are all contained inside of $P$, as can be easily checked by hand.

It remains to show that the vertices denoted with $\bullet$
{(\small%
$q_0^n = \allowbreak A^1 A_2^{2n} A_2^2 A_1 v_0$, 
$q_1^n = \allowbreak A^1 A^2 A_2^{2n} A_2^2 A_1 v_0$)}
are contained in $P$ for all $n\in\NN_0$.
We do not need to check the vertex denoted with $\circ$
{\small(%
$A_2 A_2^{2n} A_2^2 A_1 v_0$)}
since this is a covered vertex.
Denote for $m\in\{0,1\}$ the limit points of $q_0^n$ and $q_1^n$ by
$p_m = A^1 \Pi^m \Pi^{\circ2} A_2^2 A_1 v_0$.
It is easy to check that $p_0$ and $p_1$ are contained in $\co_s P$,
in particular
\begin{align*}
\max_{m\in\{0,1\}} \norm{p_m}_{\co_s P} &= 
\max \set{
\norm{p_0}_{\co_s P},\, 
\norm{p_1}_{\co_s P}}
\simeq 
\max \set{ 0.9714,\, 0.9819}
\simeq
0.9819.
\end{align*}
The matrix $A_2^2$ has the Schur form 
\begin{equation*}
A_2^2 = V D^2 V^{-1}, \quad
D^2 = \begin{bmatrix}
    1 &   &              \\
      & 1 &              \\
      &   & \frac{1}{13}
\end{bmatrix}, \quad
V = \begin{bmatrix}
    1 & 3 + \sqrt{13}  & 3 - \sqrt{13} \\
    1 & -4 - \sqrt{13} & -4 +\sqrt{13} \\
    1 & 1              & 1
\end{bmatrix}
\end{equation*}
and we set
$
\Delta^n = D^{2n}-D^{\circ 2} =
\diag{\begin{bmatrix}
    0 & 0 & \frac{1}{13^n}
    \end{bmatrix}
}
$.
Due to simple form of $\Delta$ we immediately 
obtain
$\max_{n\in\NN_0} \norm{\Delta^n}_{\co_s P} \leq 
\allowbreak
\norm{\Delta}_{\co_s P} \simeq 0.0006085$,
in the general case we need to resort to the computation of the 
pseudo spectral radius of~$\Delta$.
It remains to compute~$w^R$,
\begin{equation*}
V^{-1} A_2 A_1 v_0 = w^1 + w^R
\implies 
w^R = \frac{1}{507}\begin{bmatrix}
0&0&-78 + 23 \cdot \sqrt{13}
\end{bmatrix}^T.
\end{equation*}
Eventually, we estimate
\begin{align*}
\norm{q^n} &\leq 
\norm{A V} \cdot
    \max_{n} \norm{\Delta^{n}}  \cdot \norm{w^R} +
\max_{m \in\{0,\,1\}} \norm{p^m}
\\&
\leq
7.300 \cdot  0.0006085 \cdot 0.01299 
+
0.9819
\leq 
0.9820
\end{align*}
which proves the claim.
\end{example}

\begin{remark}\label{rem_extrapath}
The function \texttt{ipa} contained in~\cite{Mej2020} does not implement yet the ``tree flavoured``-modification,
but a simpler version of it is implemented; It tries finite subtrees in the lines marked with $(\ast)$.
This already decreases the computational time for some matrices of high dimensions significantly,
see Figure~\ref{fig:extrapath}.

\end{remark}

\begin{figure}
\centering
\begin{tabular}{r|cccccc}
    \emph{dimension} &  4  &  6   &  8  & 10  & 12  & 14  \\ 
    \emph{ratio of vertices} & 0.6 & 0.30 & 0.3 & 0.2 & 0.1 & 0.1 \\
    \emph{ratio of computation time} &  1  & 3.6  & 2.5 & 3.0 & 0.9 & 0.4
\end{tabular}
\caption{%
Comparison between the original invariant polytope algorithm,
and the algorithm described in Remark~\ref{rem_extrapath}.
For each dimension 20 tests were done.
First row: Ratio of number of vertices of computed polytope (values smaller than one mean, the new algorithm produces polytopes with less vertices).
Second row: Ratio between the computation times (values smaller than one mean, the new algorithm is faster).
One can see, that the new algorithm produces consistently smaller polytopes, but the introduced overhead does not pay of for small matrices.
}
\label{fig:extrapath}
\Description{}
\end{figure}

\section{Conclusion}
We designed two similar new algorithms for computing the $\JSR$ of a finite set of matrices,
both being a combination of the tree algorithm and the invariant polytope algorithms.
Which of the two has better performance,
and which of the two needs weaker conditions for termination
is to be investigated.

\section*{Acknowledgments}
The first author is sponsored by the Austrian Science Fund FWF, 
grant number~P3335221.

\appendix


\section{Upper and lower bounds of the polytope norm}
Unfortunately and as already noted,
the estimate of $\Delta^{n,l\circ,K}$,
as well as the computation of the pseudo spectral radius in the polytope norm,
is prohibitively expensive. Thus we have to resort to other norms in some cases.

\begin{lemma}
\label{thm:boundptnorm}
Let $V,W\varsubsetneq\RR^s$ be finite sets of vertices, 
and $\norm{\vardot}_V$, $\norm{\vardot}_W$ the corresponding polytope norms
with unit balls $\co_s V, \co_s W\varsubsetneq\RR^s$.
It holds that
\begin{equation*}
\norm{x}_W 
\leq 
\max_{v\in V} \norm{v}_W \cdot
\norm{x}_V.
\end{equation*}
In particular for the 1-norm,
\begin{equation*}
\min_{v\in V} \norm{v}_1^{-1} \cdot \norm{x}_1
\leq
\norm{x}_V
\leq
\max_{i=1,\ldots,s} \norm{e_i}_V \cdot \norm{x}_1.
\end{equation*}
\end{lemma}
\begin{proof}
This follows from the fact, that both norm's unit balls are polytopes,
and thus, we only need to compute the norms of one polytope`s vertices
in the other polytope's norm.
\end{proof}

\begin{lemma}
\label{thm:bound2norm}
Let $V=\{v_n\in\RR^s:n=1,\ldots,N\}$ be a set of vertices
and $x\in\RR^s$.
It holds that
\begin{equation}
\label{equ:bound2norm_1}
r_s \norm{x}_2 \leq
\norm{x}_{\co_s V} \leq
R_s \norm{x}_2,
\end{equation}
where $r$ and $R$ are defined as
\begin{align*}
R_s^{-1} &= \max_{v\in V} \norm{v}_2, \\
r_s^{-2}  & = \left\{ \begin{aligned}
& \minimize \norm{p}_2^2,\ p\in\RR^s,  \text{ subject to } \\
& p = \sum_{v\in V} a_v v,\\
&\sum_{v\in V} \abs{a_v} \leq 1
\end{aligned}
\right.
.
\end{align*}
In particular, for $A\in\RR^{s\times s}$,
\begin{equation}
\label{equ_bound2norm_3}
\frac{r_s}{R_s} \norm{A}_2 \leq
\norm{A}_{\co_s P} \leq 
\frac{R_s}{r_s} \norm{A}_2
\end{equation}
\end{lemma}

\section{Generalization to the complex case}
\label{sec:complex}
In general, the leading eigenvalues may be complex. 
In this case, one has to replace the symmetric convex hull $\co_s$ with the
elliptic convex hull, see Definition~\ref{def:hull} and Figure~\ref{fig:hull}.
Also the estimates about the pseudo spectral radius in Section~\ref{sec:implementation} change.
For the remaining details see~\cite{MP2022}.
\begin{definition}\label{def:hull}
    For $v=a+ib\in\CC^s$ we define its corresponding ellipse $E(v)=E(a,b)$
    as the set of points $\set{a \cos t + b \sin t:t\in\RR}$.
    For finite $V\varsubsetneq\CC^s$, we define the \emph{elliptic convex hull}
    of $V$ by
    \begin{equation}\label{equ_cosym_absco}
    \coell V = \co \set{ E(v) : v\in V}.
    \end{equation}
\end{definition}

\begin{figure}[t]
\centering
{\small

\begin{tikzpicture}[scale=0.70]
    \draw [-stealth](-3,0) -- (3,0);
    \draw [-stealth](0,-2.5) -- (0,2);
    
    \node at (1,1.5) {$v_1$};
    \node at (1,2) {\textbullet};
    \node at (2.4,1) {$v_2$};
    \node at (3,1) {\textbullet};
    \node at (3.5,1.7) {$\co_s(\{v_1,v_2\})$};
    \draw [thick] (1,2) -- (3,1) -- (-1,-2) -- (-3,-1) -- (1,2);
\end{tikzpicture}
\begin{tikzpicture}[scale=0.70]
    \draw [-stealth](-2.5,0) -- (2.5,0);
    \draw [-stealth](0,-2.5) -- (0,2.5);

    \draw    (0,0) ellipse (2 and 1);
    \node at (0.8,2.2) {$E(v_3)$};
    \draw    (0,0) ellipse (1 and 2);
    \node at (2.7,0.3) {$E(v_4)$};
    \node at (3,1.3) {$\operatorname{absco}(\{v_3,v_4\})$};
    
    \draw  [thick] ( 1.788854381999832, 0.447213595499958) -- ( 0.447213595499958, 1.788854381999832);
    \draw [thick] (-1.788854381999832, 0.447213595499958) -- (-0.447213595499958, 1.788854381999832);
    \draw [thick] ( 1.788854381999832,-0.447213595499958) -- ( 0.447213595499958,-1.788854381999832);
    \draw [thick] (-1.788854381999832,-0.447213595499958) -- (-0.447213595499958,-1.788854381999832);
\end{tikzpicture}
}
\caption[Various convex hulls]%
{
Symmetric and elliptic convex hull. 
$v_1=\begin{bsmallmatrix}1\\2\end{bsmallmatrix}$,
$v_2=\begin{bsmallmatrix}2\\1\end{bsmallmatrix}$,
$v_3=\begin{bsmallmatrix}2\\i\end{bsmallmatrix}$,
$v_4=\begin{bsmallmatrix}i\\2\end{bsmallmatrix}$,
}
\label{fig:hull}
\Description{Various convex hulls}
\end{figure}

\begin{lemma}
\label{thm:bound2norm_casec}
Let $V=\{v_n\in\CC^s:n=1,\ldots,N\}$ be a set of vertices and $x\in\CC^s$.
It holds that
\begin{equation*}
r_e \norm{x}_2 \leq
\norm{x}_{\coell P} \leq
R_e \norm{x}_2,
\end{equation*}
where $r_e$ and $R_e$ are defined as
\begin{align*}
R_e^{-1} &= \max_{v\in \Re V} \norm{v}_2 + \max_{v\in \Im V} \norm{v}_2\\
r_e^{-2}  & = \left\{ \begin{aligned}
& \minimize 2\norm{p}_2^2,\ p\in\RR^s,  \text{ subject to } \\
& \sum_{v\in \Re V \cup \Im V} a_v v = p,\\
&\sum_{v\in \Re V \cup \Im V} \abs{a_v} \leq 1
\end{aligned}
\right.
.
\end{align*}
In particular, for $A\in\RR^{s\times s}$,
\begin{equation*}
\frac{r_e}{R_e} \norm{A}_2 \leq
\norm{A}_{\coell P} \leq 
\frac{R_e}{r_e} \norm{A}_2
\end{equation*}
\end{lemma}
\begin{proof}
This follows from Lemma~\ref{thm:bound2norm} and the fact that
\begin{equation*}
\frac{1}{2} \norm{x}_{\co_s\Re V \cup \Im V}\leq 
\norm{x}_{\coell V} \leq 
\norm{x}_{\co_s \Re V \cup \Im V}
\end{equation*}
\end{proof}

{\small

\bibliographystyle{alpha}  

\newcommand{\doi}[1]{\href{https://doi.org/#1}{doi:~#1}}
\newcommand{\arxiv}[1]{\href{https://arxiv.org/abs/#1}{arXiv:~#1}}
\newcommand{\ttilde}{{\raise-1ex\hbox{\textasciitilde}}}

}

\end{document}